\documentclass{amsart}
\usepackage{amssymb,amsmath,amsrefs}
\usepackage[colorlinks=true]{hyperref}
\hypersetup{urlcolor=blue, citecolor=green}
\usepackage[dvips]{graphicx}
\usepackage{color}

\theoremstyle{plain}

\newtheorem{mainthm}{Theorem}

\newtheorem*{conj*}{Conjecture}
\newtheorem*{cor*}{Corollary}
\newtheorem{theorem}{Theorem}[section]

\newtheorem{prop}[theorem]{Proposition}
\newtheorem{proposition}[theorem]{Proposition}

\newtheorem{lemma}[theorem]{Lemma}

\newtheorem{question}{Question}

\theoremstyle{definition}
\newtheorem*{def*}{Definition}
\newtheorem{remark}[theorem]{Remark}

\newtheorem{example}[theorem]{Example}

\newtheorem{definition}{Definition}

\renewcommand{\epsilon}{\varepsilon}

\newcommand{\Z}{\mathbb{Z}}
\newcommand{\N}{\mathbb{N}}
\newcommand{\R}{\mathbb{R}}
\newcommand{\eps}{\varepsilon}

\newcommand{\dist}{\operatorname{\textit{d}}}
\newcommand{\clos}{\operatorname{clos}}

\newcommand{\diam}{\operatorname{diam}}

\title[Countably expansive homeomorphisms with the shadowing property]{Countably and entropy expansive homeomorphisms with the shadowing property.}

\author[A. Artigue, B. Carvalho, W. Cordeiro and J. Vieitez]{Alfonso Artigue, Bernardo Carvalho, Welington Cordeiro \\ and Jos\'e Vieitez}
\date{\today}


\begin{document}
\begin{abstract}
We discuss the dynamics beyond topological hyperbolicity considering homeomorphisms satisfying the shadowing property and generalizations of expansivity. It is proved that transitive countably expansive homeomorphisms satisfying the shadowing property are expansive in the set of transitive points. This is in contrast with pseudo-Anosov diffeomorphisms of the two-dimensional sphere that are transitive, cw-expansive, satisfy the shadowing property but the dynamical ball in each transitive point contains a Cantor subset. We exhibit examples of countably expansive homeomorphisms that are not finite expansive, satisfy the shadowing property and admits an infinite number of chain-recurrent classes. We further explore the relation between countable and entropy expansivity and prove that for surface homeomorphisms $f\colon S\to S$ satisfying the shadowing property and $\Omega(f)=S$, both countably expansive and entropy cw-expansive are equivalent to being topologically conjugate to an Anosov diffeomorphism.
\end{abstract}

\maketitle
\section{Introduction and Statement of Results}

The study of dynamical systems beyond the hyperbolic ones is today a very active and productive topic in the theory of differentiable dynamical systems (see \cite{BDV} for a description of the main features and results of the theory). 
A classical way to understand dynamics is through the viewpoint of topological dynamics, and a topological notion of hyperbolicity is widely accepted in the literaure: expansive homeomorphisms satisfying the shadowing propery (see Aoki and Hiraide monograph \cite{AH}). Thus, it is natural to consider generalizations of topological hyperbolicity and explore the world \emph{beyond topological hyperbolicity} (see \cite{ACCV}). In this paper, we study the dynamics of systems beyond the topologically hyperbolic ones considering homeomorphisms satisfying the shadowing property and generalizations of expansivity, and these are many: n-expansivity \cites{Art,APV,CC,CC2,LZ,Morales}, finite expansivity \cite{CC}, countable and measure expansivity \cites{Art,CDX,ArCa}, cw-expansivity \cites{Kato93,Kato93B}, entropy expansiveness \cites{Bowen,PaVi}, among others (see Section 2 for precise definitions).

The understanding of the structure of local stable and unstable sets is the main step to study the dynamics of each of these class of systems. 
In the hyperbolic systems, the Stable Manifold Theorem assures that local stable and unstable sets are manifolds, leaves of a pair of transverse foliations, that contracts and expands exponentially. In topologically hyperbolic systems local stable and unstable sets can be modeled with Hiraide's generalized foliations (see \cite{Hiraide2} and \cite{Hiraide}). In both classes, one important fact is the existence of $\eps>0$ such that the $\eps$-stable set of $x$  is contained in the stable set of $x$ and also the $\eps$-unstable set of $x$ is contained in the unstable set of $x$, that is, $W^s_{{\eps}}(x)\subset W^s(x)$ and $W^u_{{\eps}}(x)\subset W^u(x)$ for every $x\in X$ (see \cite{Ma} for a proof and Section 2 for definitions).

If we consider generalizations of topological hyperbolicity, local stable and unstable sets can be much more complicated. One example enlightening this situation is the Pseudo-Anosov diffeomorphisms of the sphere $\mathbb{S}^2$ constructed as follows. Consider an Anosov diffeomorphism $f_A$ of the torus $\mathbb{T}^2$ induced by a hyperbolic $2\times2$ matrix with integer coefficients and determinant one. The sphere $\mathbb{S}^2$ can be seen as the quotient of $\mathbb{T}^2$ by the antipodal map and then $f_A$ induces a diffeomorphism $g_A\colon \mathbb{S}^2\to\mathbb{S}^2$. The map $g_A$ constructed this way has many interesting dynamical properties that are not clear at a first glance. It is transitive, cw-expansive, satisfies the shadowing property but is not countably expansive. Indeed, it contains arbitrarily small dynamical balls containing hyperbolic horseshoes (see \cites{PaVi,AAV,ArtigueDend,Walters} for more details). Also, the local stable set of each forward transitive point contains a stable continuum and also a cantor subset contained in the unstable continuum (an in particular in the dynamical ball) as explained in the proof of Proposition 2.2.2 of \cite{ArtigueDend}.

A second class of examples with a wild behaviour on local stable and unstable sets are the n-expansive (and finite expansive) homeomorphisms with the shadowing property in \cite{CC}. These examples are constructed as the union of a topologically hyperbolic homeomorphism and a sequence of periodic orbits such that for each $\eps>0$ there exists n different periodic points with the same period in a common $\eps$-stable set. Hence, there are exactly n diffenrent stable sets intersecting the same local stable set (the same holds for the $\eps$-unstable set). These phenomena allow the existence of infinitely many distinct chain recurrent classes, in contrast with the topologically hyperbolic case. Inspired in these examples we construct examples of countably expansive homeomorphisms with the shadowing property containing a countable number of different stable sets in the same local stable set and an infinite number of chain recurrent classes (see Example \ref{Example1} in the end of the paper). %The example is basicaly the union of a topologically hyperbolic homeomorphism containing a fixed point $p$, with a sequence of fixed points $(p_n)_{n\in\N}$ converging to $p$.

For countably expansive homeomorphisms, 
the first phenomenon of the pseudo-Anosov of the sphere (explained above) is excluded by definition and the second one appears in systems with an infinite number of chain recurrent classes. Then assuming countable expansivity, transitivity and the shadowing property, we prove expansiveness in the set of transitive points. First, we let $T^+$ denote the set of forward transitive points, $T^-$ denote the set of backward transitive points, $T=T^+\cap T^-$ the set of transitive points and $\Gamma_{\eps}(x)=W^s_{{\eps}}(x)\cap W^u_{{\eps}}(x)$ denote the dynamical ball of $x$ with radius $\eps$.

\begin{mainthm}
\label{thmEstablesEnT}
If a homeomorphism $f\colon X\to X$ of a compact metric space is transitive, countably expansive and satisfies the shadowing property, then there exists $\eps>0$ such that:
\begin{enumerate}
\item $W^s_{{\eps}}(x)\subset W^s(x)$ for every $x\in T^{+}$,
\item $W^u_{{\eps}}(x)\subset W^u(x)$ for every $x\in T^{-}$ and
\item $\Gamma_{\eps}(x)=\{x\}$ for every $x\in T$.
\end{enumerate}
\end{mainthm}

%The transitivity assumption is important in Theorem \ref{thmEstablesEnT}. 
In the case $X=S$ is a compact surface we prove these systems must be topologically conjugate to Anosov diffeomorphisms. 

\begin{mainthm}\label{3.3}
For a homeomorphism of a compact surface $f\colon S\to S$ satisfying the shadowing property and $\Omega(f)=S$,
the following are equivalent:
\begin{enumerate}
 \item $f$ is conjugate to an Anosov diffeomorphism,
 \item $f$ is countably expansive,
 \item $f$ is entropy cw-expansive.
\end{enumerate}
\end{mainthm}

In \cite{APV} there are examples of 2-expansive homeomorphisms on the bitorus that are not expansive. We note that they are not transitive and do not satisfy the shadowing property. There, it is also proved that 2-expansive surface homeomorphisms without wandering points are indeed expansive. Thus, Theorem \ref{3.3} can also be viewed as a generalization of this result to the more general cases of countably and entropy cw-expansivity, when the shadowing property is present. We also remark that the identity in a cantor set is entropy cw-expansive, satisfies the shadowing property and $\Omega(f)=X$ but is not countably expansive. Hence, (2) and (3) are not equivalent for every metric space. The following is unanswered, tough.

\begin{question}
Is every countably expansive homeomorphism also entropy expansive?
\end{question}

%The following is still not clear.

%\begin{question}
%Does asymptotic expansivity imply  entropy expansivity?
%\end{question}

%As a last result, we discuss the  it is proved in \cite{CC} that topologically mixing n-expansive + shadowing + top. mixing imply L-shadowing? We know that %this implies two-sided limit shadowing and we are back again to the question whether two-sided limit shadowing implies L-shadowing, but with the additional %assumption of n-expansivity. We can use all results of \cite{CC} and also in this paper.

%The proofs of all our main theorems are contained in Section 2.

%\begin{mainthm}
%\label{corcwhImpCw1}If a (h+cw)-expansive homeomorphism $f$ of a compact metric space $X$ has the shadowing property and $\Omega(f)=X$,
%then $f$ is  asymptotically expansive.
%\end{mainthm}

%The relation between cw-expansivity and entropy expansivity is not clear. In one hand, isometries are entropy expansive but not cw-expansive and in the %other, the pseudo-Anosov diffeomorphism explained above is cw-expansive but not entropy expansive.   Particularly, the relation between countable %expansivity and entropy expansivity is not clear. The following is proved:

%\begin{mainthm} \label{C}
%If a positively $n$-expansive homeomorphism, defined in a compact metric space $X$, satisfies the shadowing property, then $X$ is finite.
%\end{mainthm}

\section{Definitions and Proofs}

Let us recall the main definitions we use in this paper. The dynamical ball of $x\in X$ of radius $c>0$ is the set $$\Gamma_{c}(x):=\{y\in X \,\,\, ; \,\,\, d(f^n(x),f^n(y))\leq c \,\,\, \text{for every} \,\,\, n\in\Z\}.$$ 
We consider the \emph{c-stable set} of $x\in X$ as the set 
$$W^s_{c}(x):=\{y\in X; \,\, d(f^k(y),f^k(x))\leq c \,\,\,\, \textrm{for every} \,\,\,\, k\geq 0\}$$
and the \emph{c-unstable set} of $x$ as the set $$W^u_{c}(x):=\{y\in X; \,\, d(f^k(y),f^k(x))\leq c \,\,\,\, \textrm{for every} \,\,\,\, k\leq 0\}.$$
The dynamical ball can be written as $$\Gamma_{c}(x)=W^u_{c}(x)\cap W^s_{c}(x).$$ 
We say that $f$ is \emph{expansive} if there exists $c>0$ such that $$\Gamma_c(x)=\{x\} \,\,\,\,\,\, \text{for every} \,\,\,\,\,\, x\in X.$$ 
We say that $f$ is n-expansive, finite expansive, countably expansive or cw-expansive if there exists $c>0$ such that $\#\Gamma_c(x)\leq n$, $\Gamma_c(x)$ is finite, $\Gamma_c(x)$ is countable, $\Gamma_c(x)$ is totally disconected, respectively, for every $x\in X$. It is clear that expansivity implies n-expansivity, that implies finite expansivity, that implies countable expansivity, which, in turn, implies cw-expansivity.

Now we define entropy expansivity. Given $n\in\N$ and $\delta>0$, we say that $E\subset X$ is $(n,\delta)$-\emph{separated}
if for each $x,y\in E$, $x\neq y$, there is $k\in \{0,\dots,n-1\}$ such that
$\dist(f^k(x),f^k(y))>\delta$.
Consider a subset $F\subset X$ and let $s_n(F,\delta)$ denote the maximal cardinality of an $(n,\delta)$-separated subset $E\subset F$.
Since $X$ is compact, $s_n(F,\delta)$ is finite.
Let\[
 h(F,\delta)=\limsup_{n\to\infty}\frac 1n\log s_n(F,\delta).
\]
We say that $f$ is \emph{entropy expansive} (or $h$-\emph{expansive}) if
there is $c>0$ such that
$h(\Gamma_c(x),\delta)=0$ for all $x\in X$ and all $\delta>0$.
Note that $h(F,\delta)$ increases as $\delta\to 0$ and define $h(F)=\lim_{\delta\to 0}h(F,\delta)$. Entropy expansivity was introduced by Bowen in \cite{Bowen} and further explored by several authors (see \cite{Buzzi} for example).

Now we define the shadowing property.  We say that a sequence $(x_k)_{k\in\Z}\subset X$ is a $\delta$-pseudo orbit if $d(f(x_k),x_{k+1})<\delta$ for every $k\in\Z$. The sequence $(x_k)_{k\in\Z}\subset X$ is $\eps$-\emph{shadowed} if there exists $y\in X$ satisfying $$d(f^k(y),x_k)<\eps, \,\,\,\,\,\, k\in\Z.$$ We say that $f$ has the \emph{shadowing property} if for every $\eps>0$ there exists $\delta>0$ such that every $\delta$-pseudo-orbit is $\eps$-shadowed.

The following is Lemma 3.1 of \cite{ACCV} and is important for the proof of Theorem \ref{thmEstablesEnT}. It was used in \cite{ACCV} as a mechanism to construct arbitrarily small topological semihorseshoes (a phenomenon similar to the one contained in the pseudo-Anosov diffeomorphism of the sphere explained in the introduction) for homeomorphisms with the L-shadowing property (a generalization of the shadowing property introduced in \cite{CC2} and explored in \cite{ACCV}).
As usual, $\Omega(f)$ denotes the set of non-wandeing points. 

\begin{lemma}\cite{ACCV}
\label{lemaWillyBarney}
Let $f\colon X\to X$ be a homeomorphism satisfying the shadowing property.
For all $\eps>0$ there is $\delta>0$ such that if $x\in\Omega(f)$, $y\in X$, $n>0$ satisfy:
\begin{equation}
 \label{ecuLiYorkeTrucho}
 \left\{
 \begin{array}{l}
 \epsilon<\max\{\dist(f^k(x),f^k(y)):0\leq k< n\}=:\gamma,\\
 \max\{\dist(x,y),\dist(f^n(x),f^n(y))\}<\delta,
 \end{array}
 \right.
\end{equation}
then
% there is an uncountable set $K\subset X$ with $h(K)>0$ and $\diam(f^k(K))\leq 2\gamma$ for all $k\in\Z$.
there are a compact set $K\subset X$ and $N\geq 1$
such that $\sup_{k\in\Z}\diam(f^k(K))\leq 2\gamma$, $f^N(K)=K$ and
$f^N\colon K\to K$ is semiconjugate to a shift.
In particular, $K$ is uncountable and $h(K)>\frac{\log(2)}N$.
\end{lemma}

The idea of the proof of Lemma \ref{lemaWillyBarney} is the following: if two-points $x,y\in X$ are $\delta$-close, one of them is a non-wandering point, some future iterates $f^k(x),f^k(y)$ are $\eps$-distant and thereafter iterates $f^n(x),f^n(y)$ become $\delta$-close again (see Figure 1 in \cite{ACCV}), then the shadowing property creates uncontable sets with positive entropy in a same dynamical ball of radius $\gamma$. 
This was used by the authors to characterize the non-expansive chain-recurrent classes of a homeomorphism satisfying the L-shadowing property.
In the following lemma, which is the main step in the proof of Theorem \ref{thmEstablesEnT}, we use Lemma \ref{lemaWillyBarney} to obtain a sufficient condition that assures the inclusion $W^s_{\alpha}(x)\subset W^s(x)$.
In the next result, $\omega(x)$ denotes the $\omega$-limit set of $x$.

\begin{lemma}
\label{lemaQuadro}
Let $f\colon X\to X$ be a countably (or entropy) expansive homeomorphism satisfying the shadowing property.
If $x\in\Omega(f)$, $z\in\omega(x)$ and $\Gamma_\alpha(z)=\{z\}$ for some $\alpha>0$,
then $W^s_\alpha(x)\subset W^s(x)$.
\end{lemma}

\begin{proof}
  If the thesis is not true,
then there exists $y\in W^s_\alpha(x)$ such that $y\notin W^s(x)$.
 Lemma \ref{lemaWillyBarney} assures the existence of $r>0$ such that
 \begin{equation}
 \label{ecuAfastados}
d(f^k(x),f^k(y))>r\text{ for every }k\in\N.
 \end{equation} Indeed, if this is not the case, then there exists an increasing sequence of positive integers $(n_k)_{k\in\N}$ such that $$d(f^{n_k}(x),f^{n_k}(y))<\frac{1}{k} \,\,\,\,\,\, \text{for every} \,\,\,\,\,\, k\in\N.$$ Since $y\notin W^s(x)$, there is $\eps>0$ and another increasing sequence of positive integers $(m_k)_{k\in\N}$ such that $$d(f^{m_k}(x),f^{m_k}(y))>\eps \,\,\,\,\,\, \text{for every} \,\,\,\,\,\, k\in\N.$$ Let $\delta>0$ be given by Lemma \ref{lemaQuadro} for this $\eps$ and choose $k\in\N$ such that $\frac{1}{k}<\delta$. Consider $i>j>k$ such that $n_k<m_j<n_i$.
Thus, $$d(f^{n_k}(x),f^{n_k}(y))<\frac{1}{k}<\delta,$$ $$d(f^{m_j}(x),f^{m_j}(y))>\eps \,\,\,\,\,\, \text{and} \,\,\,\,\,\, d(f^{n_i}(x),f^{n_i}(y))<\frac{1}{i}<\delta.$$ Since $x\in\Omega(f)$, Lemma 2.1 applies and contradicts countably (and entropy) expansivity. This proves the existence of $r>0$ satisfying (\ref{ecuAfastados}).
% Fix any $z\in X$ and take $\alpha>0$ such that $\Gamma_\alpha(z)=\{z\}$.
 Since $z\in\omega(x)$ there exists $(i_k)_{k\in\N}\subset\N$ satisfying $i_k\to\infty$ when $k\to\infty$, such that $$z=\underset{k\to\infty}{\lim}f^{i_k}(x).$$
 Since $X$ is compact, we can assume that $(f^{i_k}(y))_{k\in\N}$ converges to a point $z_y$.
 Consequently, $y\in W^s_{\alpha}(x)$ implies $$d(z_y,z)\leq\alpha.$$
 From \eqref{ecuAfastados} we obtain that $z_y\neq z$. Note that $z_y\in\Gamma_{\alpha}(z)$ since
\begin{eqnarray*}
d(f^m(z),f^m(z_y))=\underset{k\to\infty}{\lim}d(f^m(f^{i_k}(x)),f^m(f^{i_k}(y)))\leq \alpha
\end{eqnarray*}
for every $m\in\Z$.
This contradicts that $\Gamma_\alpha(z)=\{z\}$.
% If $\frac{1}{n}<\eps$ then $z_n\in\Gamma_{\eps}(z)$.
% Just note that the sequence $(z_n)_{n\in\N}$ takes an infinite number of different points of $X$ to obtain a contradiction with the fact that $c$ is the $n$-expansivity constant of $f$.
\end{proof}

The next proposition is also important in the proof of Theorem \ref{thmEstablesEnT}. The previous lemma assures that, in our scenario, the existence of a trivial dynamical ball in a point in the omega limit set of $x$ is sufficient to assure the inclusion of the local stable set of $x$ in the stable set of $x$. Then it is natural to consider the set of points with trivial dynamical ball of a suitable radius. Such points will be called \emph{expansive points} and we denote the set of all expansive points by
\[
   X_e=
  \{x\in X;\,\,\,\Gamma_\epsilon(x)=\{x\}\,\,\, \text{for some} \,\,\,\eps>0\}.
 \] If $f$ is finite expansive, then it is clear that every point is expansive and $X_e=X$. When $f$ is countable expansive the following is proved.

\begin{prop}
\label{propGammasTriviales}
 If $f$ is countably expansive, then $X_e$ is dense in $X$.
\end{prop}

\begin{proof}
Let $c>0$ be such that $\Gamma_c(x)$ is countable for all $x\in X$.
Arguing by contradiction, let $x\in X$ and $0<\delta<c$ be such that $\Gamma_\epsilon(y)\neq\{y\}$ for all
$y\in B_\delta(x)$ and all $\epsilon>0$. This implies, in particular, that $y$ is an accumulation point of $\Gamma_{\eps}(y)$.
Define $A_1=\Gamma_{\delta/2}(x)$ and
$A_n=\cup_{y\in A_{n-1}}\Gamma_{\delta/2^n}(y)$ for all $n>1$.
Notice that
\[
 A_1\subset A_2\subset\dots\subset A_n\subset \Gamma_{\delta}(x)
\]
and that each $y\in A_n$ is an accumulation point in $A_{n+1}$.
Consequently, the set
$\clos(\cup_{n\geq 1} A_n)\subset \Gamma_\delta(x)$
is perfect.
As it is compact, it is uncountable, which contradicts the countable expansivity of $f$ and the proof ends.
\end{proof}

\begin{question}
Assuming that $f$ is countably expansive, is $X_e$ residual in $X$?
\end{question}

Proposition \ref{propGammasTriviales} is the unique step in the proof of Theorem \ref{thmEstablesEnT} 
that we do not know how to prove only assuming entropy cw-expansivity. 
We say that a homeomorphism is \emph{entropy cw-expansive} if it is h-expansive and cw-expansive.

\begin{question}
Is $X_e$ non-empty for transitive entropy cw-expansive homeomorphisms satisfying the shadowing property? 
\end{question}

If this is true, then the hypothesis of countable expansivity in Theorem \ref{thmEstablesEnT} could be replaced by entropy cw-expansivity. 

\begin{remark}
Note that the identity in a cantor set is entropy cw-expansive, has the shadowing property and $\Omega(f)=X$, but $X_e=\emptyset$. The identity is not transitive, tough. Thus, transitivity should play an important role in an affirmative answer.
\end{remark}

Define the set $V^s_{c}(x)$ as the intersection $W^s(x)\cap W^s_{c}(x)$ of the stable and the local stable sets of $x$ and $V^u_{c}(x)=W^u(x)\cap W^u_{c}(x)$.

\begin{proof}[Proof of Theorem \ref{thmEstablesEnT}]
Let $f\colon X\to X$ be a transitive, countably expansive homeomorphism satisfying the shadowing property.
Proposition \ref{propGammasTriviales} assures the existence of expansive points of $f$, that is, there are
$z\in X$ and $\eps>0$ such that $\Gamma_\eps(z)=\{z\}$.
Given that $f$ is transitive, $z$ belongs to the $\omega$-limit set of every $x\in T^+$.
Then Lemma \ref{lemaQuadro} applies and we conclude that $$W^s_\eps(x)\subset W^s(x) \,\,\,\,\,\, \text{for all} \,\,\,\,\,\, x\in T^+.$$ A similar argument proves that $$W^u_\epsilon(x)\subset W^u(x) \,\,\,\,\,\, \text{for all} \,\,\,\,\,\, x\in T^-.$$
%Let $c>0$ be such that $\Gamma_c(x)$ is countable for all $x\in X$ and
%$\epsilon=c/4<\alpha$.
%From Theorem \ref{thmCountEnt}
Therefore, we have $$\Gamma_\epsilon(x)=W^s_\epsilon(x)\cap W^u_\epsilon(x)\subset V^s_{\eps}(x)\cap V^u_{\eps}(x)$$ for every $x\in T=T^+\cap T^-$.
Lemma \ref{lemaWillyBarney} again assures that $$\Gamma_\epsilon(x)=\{x\} \,\,\,\,\,\, \text{for all} \,\,\,\,\,\, x\in T.$$ Indeed, for $z\in _\epsilon(x)\cap V^u_\epsilon(x)\setminus\{x\}$, let $\eps_1=d(x,z)$, $\delta_1>0$ given by Lemma \ref{lemaWillyBarney} for $\eps_1$ and choose $k_0\in\N$ such that $$d(f^k(z),f^k(x))<\delta_1\,\,\,\,\,\, \text{whenever} \,\,\,\,\,\, |k|\geq k_0.$$ Then we obtain the exact situation where Lemma \ref{lemaWillyBarney} applies to contradict countable expansivity.
% Now we prove that $\Gamma_{\eps}(x)=\{x\}$ for every $x\in T$. If this is not the case, then there exists $x\in T$ and $y\in\Gamma_{\eps}(x)$. In particular, $$y\in W^u(x)\cap W^s(x).$$ Choose $$0<r<\frac{d(x,y)}{2}$$ and let $\delta>0$, given by the shadowing property, be such that every $\delta$-pseudo orbit is $r$-shadowed. Finally, choose $m\in\N$ such that $$d(f^k(x),f^k(y))<\frac{\delta}{4},$$ if $|k|\geq m$, and choose $l\geq m$ such that $$d(f^{-l}(x),f^m(x))<\frac{\delta}{4}.$$ This number $l$ exists because $x$ is transitive. Note that all the points $f^{-l}(x)$, $f^{-l}(y)$, $f^m(x)$ and $f^m(y)$ are $\delta$-close to each other. For each $i\in\N$ we consider the $\delta$-pseudo orbit $(x_k^i)_{k\in\Z}$ defined by
% $$x_k^i=\begin{cases}f^k(x), &k<m\\
% 			       f^{k\hspace{+0.1cm}mod(2m+l)-m-l}(x), & m\leq k<il+(1+i)m\\
% 			       f^{k-(1+i)l-(1+i)m}(y), & k\geq il+(1+i)m.\end{cases}$$
% This sequence is formed by the orbit of $x$ until $f^m(x)$, then follow $i$ copies of the segment of orbit from $f^{-l}(x)$ to $f^m(x)$, and finally the future orbit of $f^{-l}(y)$.
% The shadowing property assures, for each $i\in\Z$, the existence of $z_i\in X$ that $r$-shadows $(x^i_k)_{k\in\Z}$. Hence, the sequence $(z_i)_{i\in\N}$ belongs to the same dynamical ball of radius $3\eps$, which is smaller than the $n$-expansivity constant $c$. We just note that $z_i\neq z_j$ for every $i\neq j$ since $d(f^{j(m+l)}(z_i),y)<r$ and $d(f^{j(m+l)}(z_j),x)<r$ when $i<j$. But this contradicts the fact that $c$ is an $n$-expansivity constant of $f$.
\end{proof}

Now we turn our attention to the proof of Theorem \ref{3.3} where we prove that for a transitive homeomorphism of a compact surface satisfying the shadowing property, countable expansivity, entropy expansivity and being conjugate to an Anosov diffeomprhisms are equivalent. In the proof of this theorem, we explore the relation between countable expansivity and entropy expansivity. It is not clear whether countably expansive homeomorphisms are also entropy cw-expansive. We prove that this is true under the hypothesis of shadowing and $\Omega(f)=X$.

\begin{theorem}
\label{thmCountEnt}
If $f$ is countable expansive, $\Omega(f)=X$ and $f$ has the shadowing property, then $f$ is entropy expansive.
\end{theorem}

\begin{proof}
Suppose that $c>0$ is such that $\Gamma_c(x)$ is countable for all $x\in X$.
We will prove that $c/2$ is an entropy expansivity constant of $f$.
Arguing by contradiction, suppose that $h(\Gamma_{c/2}(x),\epsilon)>0$ for some $x\in X$ and some $\epsilon>0$.
For this value of $\epsilon$ consider $\delta>0$
given by Lemma \ref{lemaWillyBarney}. As $X$ is compact, there is $l_1\in\N$ such that every subset $E\subset X$ with $l_1$ points contains
at least two different points at distance smaller than $\delta$.
Also, there is a larger $l_2\in\N$ such that every subset $E\subset X$ with $l_2$ points contains
a subset $E'$ with $l_1$ points and with $\diam(E')<\delta$.

Since $h(\Gamma_{c/2}(x),\epsilon)>0$ we have that $\{s_n(\Gamma_{c/2}(x),\epsilon):n\in\N\}$ is unbounded.
Therefore, for some $n\in\N$,
there is an $(n,\epsilon)$-separated subset $E\subset \Gamma_{c/2}(x)$ with $l_2$ points.
Thus, there is a subset $E'\subset E$ with $l_1$ points and with $\diam(E')<\delta$.
The set $f^n(E')$ contains $l_1$ points, and two different points of $f^n(E')$ are at a distance smaller than $\delta$.
That is, there are $x,y\in\Gamma_{c/2}(x)$ such that $\dist(x,y)<\delta$, $\dist(f^n(x),f^n(y))<\delta$ and
$$\epsilon<\sup\{\dist(f^k(x),f^k(y)):0\leq k\leq n\}=\gamma\leq c/2.$$
Then, by Lemma \ref{lemaWillyBarney}, there is an uncountable set $K$ with $\diam(f^i(K))\leq c$ for all $i\in\Z$.
But this contradicts that $c$ is a countable expansivity constant, and the proof ends.
\end{proof}

%We state the general case as a question.

%\begin{question}
%Is every countably expansive homeomorphism also entropy expansive?
%\end{question}

The set $V^s_c(x)\cap V^u_c(x)$ will be called the \emph{asymptotic dynamical ball of $x$} of radius $c$. Expansivity requires that all the dynamical balls of radius $c$ are trivial, but we consider a more general notion of expansiveness where only the asymptotic dynamical balls of radius $c$ are trivial.

\begin{definition}
A homeomorphism $f\colon X\to X$ is called \emph{asymptotically expansive} if there exists $c>0$ such that $V^s_c(x)\cap V^u_c(x)=\{x\}$ for every $x\in X$.
\end{definition}

The examples of n-expansive homeomorphisms with the shadowing property in \cite{CC} are not expansive but are asymptotically expansive. Indeed, there are periodic points in arbitrarily small dynamical balls, but since distinct periodic points belong to distinct stable and unstable sets, they do not belong to the same asymptotic dynamical ball.

\begin{remark}
We remark that for homeomorphisms satisfying the L-shadowing property, expansivity and asymptotic expansivity are equivalent as the argument of Lemma 3.3 in \cite{ACCV} shows.
\end{remark}

The expansiveness on transitive points in Theorem \ref{thmEstablesEnT} might not be extended to the whole space, but we can prove asymptotic expansiveness even in the more general scenario of entropy expansivity.

\begin{proposition}\label{ass}
If a homeomorphism $f\colon X\to X$ of a compact metric space is entropy expansive, satisfies the shadowing property and $\Omega(f)=X$, then $f$ is asymptotically expansive.
\end{proposition}

\begin{proof}
If $f$ is not asymptotically expansive, then for every $\eps>0$ there exists $x\in X$ and $y\in V^s_{\eps}(x)\cap V^u_{\eps}(x)\setminus\{x\}$. As in the proof of Theorem \ref{thmEstablesEnT}, Lemma \ref{lemaWillyBarney} contradicts the entropy expansivity.
\end{proof}

%\begin{corollary}
%\label{corcwhImpCw1}If a (h+cw)-expansive homeomorphism $f$ of a compact metric space $X$ has the shadowing property and $\Omega(f)=X$,
%then $f$ is  asymptotically expansive.
%\end{corollary}

%\begin{proof}
%\azul{bla} 
%\end{proof}

%\begin{corollary}\label{3.3}
%For a transitive homeomorphism $f\colon S\to S$ with the shadowing property of a compact surface,%
%%the following are equivalent:
%\begin{enumerate}
% \item $f$ is conjugate to an Anosov diffeomorphism,
 %\item $f$ is countable expansive,
% \item $f$ is (h+cw)-expansive.
%\end{enumerate}
%\end{corollary}

Before proving Theorem \ref{3.3} we recall the definition of cwN-expansivity.
We say that $f$ is \emph{cw}$N$-\emph{expansive} if there is $\epsilon>0$ such that
if $C^s$ and $C^u$ are $\epsilon$-stable and $\epsilon$-unstable continua, respectively, then $C^s\cap C^u$ contains at most $N$ points.
Expansive homeomorphisms are always cw1-expansive, while the pseudo-Anosov diffeomorphism described in the introduction is cw2-expansive.
These properties were introduced in \cite{ArtigueDend} where, in particular, the relation of cw1-expansivity and expansivity was discussed.
It is proved in \cite{ArtigueDend}*{Theorem 6.8.5} that cw1-expansive surface homeomorphisms are indeed expansive.

\vspace{+0.4cm}

\begin{proof}[Proof of Theorem \ref{3.3}]
 Every homeomorphism conjugate to an Anosov diffeomorphism is expansive, and in particular countably expansive.
 By Theorem \ref{thmCountEnt}, a countably expansive homeomorphism satisfying the shadowing property and $\Omega(f)=S$ is entropy expansive (and every countably expansive homeomorphism is cw-expansive).
To close the cycle, we have by Proposition \ref{ass} that entropy expansive homeomorphisms satisfying the shadowing property and $\Omega(f)=S$ are asymptotically expansive.
Note that asymptotically expansive homeomorphisms that are cw-expansive are also cw1-expansive, since the existence of $\eps$-stable and unstable continua such that $C^s\cap C^u$ contains at least two points $x,y$ and \cite{Kato93}*{Proposition 2.1} that assures local stable (unstable) continua are contained in the same stable (unstable) set, imply that $y\in V^s_{\eps}(x)\cap V^u_{\eps}(x)$ contradicting asymptotic expansivity. Then we conclude that $f$ is expansive by \cite{ArtigueDend}. It is known that an expansive homeomorphism with the shadowing property on compact surfaces are topologically conjugate to an Anosov diffeomorphism (see \cite{Hiraide2}). This finishes the proof.
\end{proof}

\begin{example}
\label{Example1}
We define a countably expansive homeomorphism in a compact metric space satisfying the shadowing property that is not finite expansive and contains infinitely many distinct chain recurrent classes. Consider an expansive homeomorphism $g$ defined in a compact metric space $(M,d_0)$ and satisfying the shadowing property. Suppose further that $g$ admits a fixed point $p_0$. Define $X=M\cup E$, where $E$ is any infinite enumerable set. In this case, let $(p_k)_{k\in\N}$ be an enumeration of $E$. Define a function $d\colon X\times X\to\R$ by 
$$d(x,y)=\begin{cases}0, & x=y \\
d_0(x,y), & x,y\in M \\
\frac{1}{k}+d_0(x,p_0), & x\in M, y=p_k \\
\frac{1}{k}+d_0(y,p_0), & x=p_k, y\in M \\
\frac{1}{m}+\frac{1}{k}, & x=p_m, y=p_k.
\end{cases}$$

One can prove that $(X,d)$ is a compact metric space. We invite the reader to fill in the details, adapting the argument of \cite{CC}, Section 3. Now we define a function $f\colon X\to X$ by
$$f(x)=\begin{cases} g(x), & x\in M \\
x, & x\in E.
\end{cases}$$ Then $f$ is a countable expansive homeomorphism. Indeed, the dynamical ball for $g$ of any point of $M$ is trivial, since $g$ is expansive, and $X$ is the union of $M$ with a countable set $E$, which could just increase the dynamical ball of any point of $M$ to a countable subset of $X$. Moreover, the points of $E$ are fixed points of $f$ that are close to the fixed point $p_0$. Then no point of $M\setminus\{p_0\}$ could be in the dynamical ball of $p_k$, since it would also be in the dynamical ball of $p_0$, that is trivial in $M$. This proves that the dynamical ball of any point in $E$ is contained in $E\cup\{p_0\}$, that is countable.

We note that $f$ is not finite expansive, since for each $\eps>0$, the set $\Gamma_{\eps}(p_0)$ contains an infinite number of the fixed points $p_k\in E$. Moreover, $f$ has the shadowing property. Indeed, pseudo-orbits contained in $M$ are shadowed because $g$ has the shadowing property. If a pseudo-orbit contains points of $E$ then we consider a pseudo orbit in $M$ switching any point in $E$ by $p_0$. Shadowing orbits of this pseudo orbit also shadow the original pseudo orbit and the shadowing property is proved. Finally, it is easy to see that each $p_k\in E$ belongs to a different chain recurrent class, so that $f$ admits an infinite number of chain recurrent classes.
\end{example}

\section*{Acknowledgements}
The second author was supported by Capes, CNPq and the Alexander von Humboldt Foundation. Part of this work was developed while the second author was visiting the Departamento de Matem\'atica y Estad\'\i stica del Litoral in Salto, Uruguay, where some conversations with Mauricio Achigar happened.

\vspace{1.5cm}
\noindent

{\em A. Artigue and J. Vieitez}
\vspace{0.2cm}

\noindent

Departamento de Matem\'atica y Estad\'istica del Litoral,

Universidad de la Rep\'ublica,

Gral. Rivera 1350, Salto, Uruguay
\vspace{0.2cm}

\email{artigue@unorte.edu.uy}

\email{jvieitez@unorte.edu.uy}

\vspace{1.5cm}
\noindent

{\em B. Carvalho}
\vspace{0.2cm}

\noindent

Departamento de Matem\'atica,

Universidade Federal de Minas Gerais - UFMG

Av. Ant\^onio Carlos, 6627 - Campus Pampulha

Belo Horizonte - MG, Brazil.
\vspace{0.2cm}

Friedrich-Schiller-Universität Jena

Fakultät für Mathematik und Informatik

Ernst-Abbe-Platz 2

07743 Jena

\vspace{0.2cm}

\email{bmcarvalho@mat.ufmg.br}

\vspace{1.5cm}
\noindent

{\em W. Cordeiro}

\noindent

Institute of Mathematics, Polish Academy of Sciences

ul. \'Sniadeckich, 8

00-656 Warszawa - Poland

\vspace{0.2cm}

\email{wcordeiro@impan.pl}

\end{document}